\documentclass[12pt,reqno]{amsart}
\usepackage{graphicx}
\usepackage{amssymb,amsmath}
\usepackage{amsthm}
\usepackage{color,graphicx}
\usepackage{hyperref}
\usepackage{color}
\usepackage{epstopdf}

\newcommand{\be}{\begin{equation}}

\newcommand{\ee}{\end{equation}}

\newcommand{\dn}{{\rm \,dn}}

\newcommand{\R}{{\mathbb R}}

\newcommand{\ve}{{\varepsilon}}

\numberwithin{equation}{section}
\numberwithin{figure}{section}

\newtheorem{theorem}{Theorem}[section]
\newtheorem{proposition}[theorem]{Proposition}
\newtheorem{remark}[theorem]{Remark}
\newtheorem{lemma}[theorem]{Lemma}

\newtheorem{definition}[theorem]{Definition}

\begin{document}
\vglue-1cm \hskip1cm
\title[Periodic Waves for the Modified Kawahara Equation]{On the Stability of Periodic Traveling Waves for the Modified Kawahara Equation}

\begin{center}



\maketitle

{\bf Gisele Detomazi Almeida}

{Universidade Federal do Tocantins - Campus de Arraias\\
Av. Universit\'aria, s/n, CEP 77330-000, Arraias, TO, Brazil.}\\
{gisele@uft.edu.br}

\vspace{1mm}

{\bf Fabr\'icio Crist\'ofani }

{IMECC-UNICAMP\\
	Rua S\'ergio Buarque de Holanda, 651, CEP 13083-859, Campinas, SP,
	Brazil.}\\
{ fabriciocristofani@gmail.com}

\vspace{1mm}

{\bf F\'abio Natali}

{Departamento de Matem\'atica - Universidade Estadual de Maring\'a\\
	Avenida Colombo, 5790, CEP 87020-900, Maring\'a, PR, Brazil.}\\
{fmanatali@uem.br}

\end{center}

\begin{abstract}
In this paper, we present the first result concerning the orbital stability of periodic traveling waves for the modified Kawahara equation. Our method is based on the Fourier expansion of the periodic wave in order to know the behaviour of the nonpositive spectrum of the associated linearized operator around the periodic wave combined with a recent development which significantly simplifies the obtaining of orbital stability results. 
\end{abstract}

\section{Introduction}

The orbital stability of periodic traveling-wave solutions associated to the modified Kawahara equation
\be\label{horder}
u_t+u^2u_x+\gamma u_{xxx}-u_{xxxxx}=0,
\ee
will be shown in this paper. Here, $\gamma\geq0$, $u:\R\times\R\to\R$ is a real spatially $L$-periodic function. Equation $(\ref{horder})$ models wave propagation on a nonlinear transmission line (see \cite{Kano}). \\
\indent Formally, equation $(\ref{horder})$ admits the conserved quantities
\begin{equation}\label{Eu}
	P(u)=\frac{1}{2}\int_{0}^{L}\Big(u_{xx}^2+\gamma u_x^2-\frac{1}{6}u^4\Big)dx,
\end{equation}
\begin{equation}\label{Mu}
	F(u)=\frac{1}{2}\int_{0}^{L}u^2dx\ \ \ \mbox{and} \ \ \ \ M(u)=\int_{0}^{L}u\,dx.
\end{equation}

A traveling wave solution for \eqref{horder} is a solution of the form $u(x,t)=\phi(x-\omega t)$, where $\omega$ is a real constant representing the wave speed and $\phi:\R\to\R$ is a periodic function. Substituting this form into (\ref{horder}), we obtain
\begin{equation}\label{ode-wave}
\phi''''-\gamma\phi''+\omega\phi-\frac{1}{3}\phi^3+A=0,
\end{equation}
where $A$ is a constant of integration.

In view of the conserved quantities (\ref{Eu})-(\ref{Mu}), we may define the augmented Lyapunov functional,
\begin{equation}\label{lyafun}
	G(u)=P(u)+\omega F(u)+AM(u),
\end{equation}
and the linearized operator around the wave $\phi$,
\begin{equation}\label{operator}
	\mathcal{L}:=G''(\phi)=\partial_x^4-\gamma\partial_x^2+\omega-\phi^2.
\end{equation}
In particular, we see that $\phi$ is a critical point of $G$.\\
\indent Now, we present some contributors concerning the orbital stability of explicit periodic/solitary waves related 
to the generalized Kawahara equation

\be\label{horder1}
u_t+u^pu_x+\gamma u_{xxx}-u_{xxxxx}=0,
\ee
where $p\geq1$ is an integer. In fact, for the case $p=1$, the authors in \cite{ACN} established the orbital stability of explicit periodic traveling waves solution of the form

\begin{eqnarray}\label{sol}
\phi(x) = a &+& b\left(\mbox{dn}^2\left(\frac{2K}{L}x,k\right)-\frac{E}{K}\right)   \nonumber \\
&+& d\left(\mbox{dn}^4\left(\frac{2K}{L}x,k\right)-(2-k^2)\frac{2E}{3K}+\frac{1-k^2}{3}\right), \label{sol1}
\end{eqnarray}
where $a$, $b$ and $d$ are real parameters. Here $\dn$ represents the Jacobi elliptic function of  dnoidal type, $K=K(k)$ is the complete elliptic integral of the first kind, $E=E(k)$ is the complete elliptic integral of the second kind and both of them depend on the elliptic modulus $k\in(0,1)$ (see \cite{bird} for additional details). The method used in \cite{ACN} to obtain the stability was an adaptation of the method in \cite{AP}. \\
\indent Regarding the orbital stability of solitary waves, Albert \cite{albert1} determined that the solitary wave
$\phi(x)=\mbox{sech}^4(bx)$, where $b$ is uniquely determined by a single wave-speed $\omega=\omega_0$ (there is no smooth curve of explicit solitary waves), is a stable solution for the equation $(\ref{horder1})$ with $p=1$. To do so, the author used the orthogonality of Gegenbauer polynomials to prove that the quantity
$\mathcal{I}=\langle\mathcal{L}\Phi,\Phi\rangle$ is strictly negative (see Lemma $\ref{prop2}$).

\indent The explicit periodic solution for the equation $(\ref{ode-wave})$ is given by 

\begin{equation}\label{solkawa1}
\phi(x)=a + b\left[\dn^2\left(\frac{2K(k)}{L}x,k\right)-\frac{E(k)}{K(k)}\right],
\end{equation}
where $a$ and $b$ depends smoothly on the wave speed $\omega$.\\ 
\indent Since the explicit solution in $(\ref{solkawa1})$ is determined, it is possible to use the approach in \cite{AN} in order to determine the behaviour of the nonpositive spectrum of $\mathcal{L}$ in $(\ref{operator})$. With the previous knowledge of the nonpositive spectrum of $\mathcal{L}$, we are able to use the recent developments in \cite{ANP} and \cite{Stuart} to establish a result of orbital stability for positive and periodic solutions associated to the equation $(\ref{horder})$. \\
\indent Next section is devoted to present the basic framework of stability. In Section 3, we present the existence of a positive solution for the equation $(\ref{ode-wave})$ and our result of orbital stability in the energy space $H_{per}^{2}([0,L])$.

\section{Basic Framework of Orbital Stability}\label{OSPW}

We follow the arguments contained in \cite{ANP} in order to present a basic framework of the orbital stability of periodic waves. Consider $u$ and $v$ in the energy space $X:=H_{per}^2([0,L])$, we define $\rho$ the ``distance'' between $u$ and $v$ as
	$\rho(u,v)=\inf_{y\in\mathbb{R}}||u-v(\cdot+y)||_{X}.$ Roughly speaking the distance between $u$ and $v$ is measured trough the distance between $u$ and the orbit of $v$, generated by translations.

Our precise definition of orbital stability is given below.
\begin{definition}\label{defstab}
	We say that an $L$-periodic solution $\phi$ is orbitally stable in $X$, by the periodic flow of \eqref{horder},  if for any $\ve>0$ there exists $\delta>0$ such that for any $u_0\in X$ satisfying $\|u_0-\phi\|_X<\delta$, the solution $u(t)$ of \eqref{horder} with initial data $u_0$ exists globally and satisfies
	$
	\rho(u(t),\phi)<\ve,
	$
	for all $t\geq0$.
\end{definition}

\begin{remark}
The Cauchy problem associated to the evolution equation $(\ref{horder})$ is locally well-posed in the energy space $H_{per}^2([0,L])$  according with the recent development in \cite{kwak}. The global well-posedness in the same space can be obtained by combining the local theory with the Gagliardo-Nirenberg inequality.
\end{remark}

\indent In what follows, we define the smooth functional $Q:X\rightarrow\mathbb{R}$ given by 
\begin{equation}\label{functQ}
Q(u)=\nu F(u)+\mu M(u),
\end{equation}
where $\mu$ and $\nu$ are real parameters which will be chosen later.\\
\indent For a given $\varepsilon>0$, we define the $\varepsilon$-neighborhood of the orbit $O_\phi=\{\phi(\cdot+y), y\in\R\}$ as
$
	U_{\varepsilon} := \{u\in X;\ \rho(u,\phi) < \varepsilon\}.
$
We also set  
$\Upsilon_0=\{u\in X;\ \langle Q'(\phi),u\rangle=0\},
$
where $\langle\cdot,\cdot\rangle$ denotes the scalar product in $L^2_{per}([0,L])$.
Note that $\Upsilon_0$ is exactly the tangent space to $\{u\in X; Q(u)=Q(\phi)\}$ at $\phi$.

In order to prove the desired stability we follow the strategy put forward in \cite{CNP}, \cite{NP1}, and \cite{Stuart}. Let us start by showing that $\mathcal{L}$ is strictly positive when restricted to the space $\Upsilon_0\cap \{\phi'\}^\perp$. To do so, we need to assume the following hypothesis:

\medskip

\begin{flushleft}
	$(H)$ There exists a $L$-periodic solution $\phi \in C^{\infty}_{per}([0,L])$ of (\ref{ode-wave}) with fixed period $L>0$. Moreover, the self-adjoint operator $\mathcal{L}$ has only one negative eigenvalue which is simple and zero is a eigenvalue whose eigenfunction is $\phi'$.
\end{flushleft}

\medskip

\begin{lemma}\label{prop2}
		Suppose that assumption $(H)$ occurs.  Let $\Phi$ be a smooth function such that 
		$\langle\mathcal{L}\Phi,\varphi\rangle=0$, for all $\varphi\in \Upsilon_0$ and
	$\mathcal{I}:=\langle\mathcal{L}\Phi,\Phi\rangle<0.$
	Thus, there exists $c>0$ such that
	$\langle\mathcal{L}v,v\rangle\geq c||v||_{X}^2,$
	for all $v\in \Upsilon_0\cap \{\phi'\}^\perp$.
\end{lemma}
\begin{proof}
	See Proposition 4.12 in \cite{CNP}.
\end{proof}

Lemma \ref{prop2} is useful to establish  the following result.

\begin{lemma}\label{lemma1}
	Under assumption of the Lemma $\ref{prop2}$, there exist $N>0$ and $\tau>0$ such that
	$\langle\mathcal{L}v,v\rangle +2N\langle Q'(\phi),v\rangle^{2}\geq \tau||v||_{X}^2,$
	for all $v\in \left\{\phi'\right\}^{\perp}$.
\end{lemma}
\begin{proof}
Given $v\in \left\{\phi'\right\}^{\perp}$, let us define
	$z=v-\zeta w,$
where $w=\frac{Q'(\phi)}{||Q'(\phi)||_{L^2_{per}}}$ and $\zeta=\langle v,w\rangle$. Since  $\langle Q'(\phi),\phi'\rangle=0$, it is easy to see that $z\in\Upsilon_0\cap \left\{\phi'\right\}^{\perp}$. Thus, Lemma \ref{prop2} implies
\begin{equation}\label{eq01}
		\langle\mathcal{L}v,v\rangle \geq \zeta^2\langle\mathcal{L}w,w\rangle + 2\zeta\langle\mathcal{L}w,z\rangle +c||z||_X^2.
	\end{equation}

\indent Using Cauchy-Schwartz and Young's inequalities, we have 
	\begin{equation}
		|2\zeta\langle\mathcal{L}w,z\rangle| \leq \frac{c}{2}||z||_X^2 + \frac{2\zeta^2}{c}||\mathcal{L}w||_X^2, \nonumber
	\end{equation}
	that is,
	\begin{equation}\label{eq02}
	2\zeta\langle\mathcal{L}w,z\rangle \geq -\frac{c}{2}||z||_X^2 - \frac{2\zeta^2}{c}||\mathcal{L}w||_X^2.
	\end{equation}
	
	Choosing $N>0$ depending only on $\phi$ such that
	\begin{equation}\label{eq03}
		\langle\mathcal{L}w,w\rangle -\frac{2}{c}||\mathcal{L}w||_X^2 +2N||Q'(\phi)||_{L^2_{per}}^2\geq \frac{c}{2}\|w\|_{X}^2,
	\end{equation}
we obtain, using (\ref{eq01})-(\ref{eq03}) the following inequality
\begin{eqnarray}
\langle\mathcal{L}v,v\rangle +2N\langle Q'(\phi),v\rangle^{2}&=&\langle\mathcal{L}v,v\rangle + 2N\zeta^2||Q'(\phi)||_{L^2_{per}}^2 \nonumber \\
&\geq&\frac{c}{2}(\zeta^2\|w\|^2_X+||z||_X^2) \nonumber \\
&\geq& \tau||v||_{X}^2, \nonumber
\end{eqnarray}
where $\tau>0$ is a constant that does not depend on $v$. The proof is thus completed.
\end{proof}

Let $N>0$ be the constant obtained in the previous lemma. We define the modified Lyapunov functional $V:X\rightarrow\R$ as
$$V(u)=G(u)-G(\phi)+N(Q(u)-Q(\phi))^2,$$ where $G$ is the functional defined in (\ref{lyafun}). It is easy to see from $(\ref{ode-wave})$ that $V(\phi)=0$ and $V'(\phi)=0$.

\begin{lemma}\label{lemma2}
	Under assumptions of the Lemma $\ref{prop2}$, there exist $\alpha>0$ and $D>0$ such that 
	$V(u)\geq D\rho(u,\phi)^2,$
	for all $u\in U_{\alpha}$.
\end{lemma}
\begin{proof}
	First, note that from the definition of $V$ it follows that
	$$\langle V''(u)v,v\rangle=\langle G''(u)v,v\rangle+2N(Q(u)-Q(\phi)) \langle Q''(u)v,v\rangle+2N\langle Q'(u),v\rangle^2,$$
	for all $u,v\in X$. In particular, 
	$\langle V''(\phi)v,v\rangle=\langle \mathcal{L}v,v\rangle+2N\langle Q'(\phi),v\rangle^2.$
	Consequently, from Lemma \ref{lemma1} we get
	\begin{equation} \label{eq001}
	\langle V''(\phi)v,v\rangle\geq\tau||v||_X^2,
	\end{equation}
	for all $v\in \left\{\phi'\right\}^{\perp}$.
	
	On the other hand, a Taylor expansion of $V$ around $\phi$  reveals that
	\begin{equation}\label{eq002}
	V(u)=V(\phi)+ \langle V'(\phi),u-\phi\rangle+\frac{1}{2} \langle V''(\phi)(u-\phi),u-\phi\rangle +h(u),
	\end{equation}
	where $\lim\limits_{u\to\phi}\frac{h(u)}{||u-\phi||_X^2}=0$.
	Thus, we can choose $\alpha_1>0$ such that
	\begin{equation}\label{limit1}|h(u)|\leq\frac{\tau}{4}||u-\phi||_X^2, \qquad \mbox{for all}  \ u\in B_{\alpha_1}(\phi),\end{equation}
	where $B_{\alpha_1}(\phi)=\left\{u\in X; ||u-\phi||_X <\alpha_1 \right\}$.
	
	Since $V(\phi)=0$ and $V'(\phi)=0$, we have from $(\ref{eq001})-(\ref{limit1})$ that
	$
	V(u)\geq \frac{\tau}{4}\rho(u,\phi)^2,
$
	for all $u\in B_{\alpha_1}(\phi)$ such that $(u-\phi)\in \left\{\phi'\right\}^{\perp}.$
	
	Now, let us define the smooth map $S:X\times \mathbb{R}\rightarrow\mathbb{R}$ given by $S(u,r)=\langle u(\cdot-r),\phi'\rangle$. Since $S(\phi,0)=0$ and 
	$\frac{\partial S}{\partial r}(\phi,0)=-\langle\phi',\phi'\rangle\neq0$, we guarantee, from the implicit function theorem, the existence of $\alpha_2>0$, $\delta_0>0$ and a unique $C^1-$map $r:B_{\alpha_2}(\phi)\rightarrow(-\delta_0,\delta_0)$ such that $r(\phi)=0$ and $S(u,r(u))=0$, for all $u\in B_{\alpha_2}(\phi)$. Consequently, $(u(\cdot-r(u))-\phi)\in \left\{\phi'\right\}^{\perp}$, for all $u\in B_{\alpha_2}(\phi)$. The remainder of the proof follows from similar arguments as in \cite{ANP} (see also \cite{CNP}).

\end{proof}

The above lemma is the key point to prove our main result. Roughly speaking, it says that $V$ is a suitable Lyapunov function to handle with our problem. Finally, we present the stability result.

\begin{theorem}\label{teoest} Suppose that the assumptions contained in Lemma $\ref{prop2}$ occur, then $\phi$ is orbitally stable in $X$ by the periodic flow of   $(\ref{horder})$.
\end{theorem}

\begin{proof}
	Let $\alpha>0$ be the constant such that Lemma \ref{lemma2} holds.  Since $V$ is continuous at $\phi$, for a given $\varepsilon>0$, there exists $\delta\in (0,\alpha)$ such that if $||u_0-\phi||_X<\delta$  one has
$V(u_0)=V(u_0)-V(\phi)<D\varepsilon^2,$ where $D>0$ is the constant in Lemma \ref{lemma2}.
	
	The continuity in time of the function $\rho(u(t),\phi)$ allows to choose $T>0$ such that 
	\be\label{subalpha}\rho(u(t),\phi)<\alpha,\ \ \  \mbox{for all}\ t\in [0,T).\ee
	Thus, one obtains $u(t)\in U_{\alpha}$, for all $t\in[0,T)$. Combining Lemma \ref{lemma2} and the fact that  $V(u(t))=V(u_0)$ for all $t\geq0$, we have
	\be\label{estepsilon1}
	\rho(u(t),\phi)<\varepsilon,\ \ \ \ \ \mbox{for all}\ t\in[0,T).
	\ee
	Next, we prove that $\rho(u(t),\phi)<\alpha$, for all $t\in [0,+\infty)$, from which one concludes the orbital stability. Indeed, let  $T_1>0$ be the supremum of the values of $T>0$ for which $(\ref{subalpha})$ holds. To obtain a contradiction, suppose that $T_1<+\infty$.  By choosing $\varepsilon<\frac{\alpha}{2}$ we obtain, from $(\ref{estepsilon1})$ that $\rho(u(t),\phi)<\frac{\alpha}{2},$ for all $t\in[0,T_1)$. Since $t\in(0,+\infty)\mapsto\rho(u(t),\phi)$ is continuous, there is $T_0>0$ such that $\rho(u(t),\phi)<\frac{3}{4}\alpha<\alpha$, for $t\in [0,T_1+T_0)$, contradicting the maximality of $T_1$. Therefore, $T_1=+\infty$ and the theorem  is established.
\end{proof}

\section{Stability of Periodic Waves for the Equation $(\ref{horder})$}\label{applic}

In this section, we apply the arguments developed in Section \ref{OSPW} in order to obtain the orbital stability of periodic waves for the model $(\ref{horder})$ for the case $\gamma=0$. For the sake of completeness of the reader, we rewrite the model $(\ref{horder})$ in a short equation as
\begin{equation}\label{mkawa}
u_t+u^2u_x-u_{xxxxx}=0.
\end{equation}

By looking for periodic traveling wave solutions having the form $u(x,t)=\phi(x-\omega_0 t)$, we get from $(\ref{mkawa})$ (after integration) that $\phi$ solves the nonlinear ordinary differential equation
\begin{equation}\label{edokawa}\phi''''+\omega\phi-\frac{1}{3}\phi^3+A=0.
\end{equation}

 As we have already mentioned in the introduction, equation $(\ref{edokawa})$ admits an explicit $L$-periodic solution given by the ansatz (see \cite{parkes})
\begin{equation}\label{solkawa}
\phi(x)=a + b\left[\dn^2\left(\frac{2K(k)}{L}x,k\right)-\frac{E(k)}{K(k)}\right],
\end{equation}
where
\begin{equation}\label{adn}
a=\frac{8\sqrt{10}K(k)\left[K(k)k^2-2K(k)+3E(k)\right]}{L^2}\ \ \mbox{and}\ \ \ b=\frac{24\sqrt{10}K(k)^2}{L^2}.
\end{equation}
Moreover, it is to be pointed out that $k$ is a free parameter and establishes a smooth curve of periodic solutions for \eqref{edokawa} $k\mapsto \phi:=\phi_{(\omega(k),A(k))}$ such that
$$
\omega=\frac{384(k^4-k^2+1)K(k)^4}{L^4}\ \ \mbox{and}\ \ \
A=\frac{-2048\sqrt{10}}{3}\frac{(k^2-2)K(k)^6(2k^2-1)(k^2+1)}{L^6}.
$$

\begin{remark}
The reason to consider $\gamma=0$ in equation $(\ref{horder})$ is because the case $\gamma\neq0$ produces a similar periodic traveling wave solution as determined in $(\ref{solkawa})$ with a close profile depending on the Jacobi elliptic function of \textit{dnoidal}-type. This new periodic wave has more complicated constants $a$, $b$, $\omega$ and $A$ as above depending on the modulus $k\in(0,1)$ and the period $L>0$. Indeed, by simplicity let us consider $\gamma=1$. In this case, we have 
\begin{equation}\label{solkawa2}
\phi(x)=a + b\left[\dn^2\left(\frac{2K(k)}{L}x,k\right)-\frac{E(k)}{K(k)}\right] 
\end{equation}
where
\begin{equation*}
a=\frac{\sqrt{10}}{10}\frac{80\left(k^2-2\right)K(k)^2+240E(k)K(k)}{L^2}+\frac{\sqrt{10}}{10},\ \ b=\frac{24\sqrt{10}K(k)^2}{L^2}, \nonumber
\end{equation*}
and
\begin{equation*}
\omega=\frac{3840(k^4-k^2+1)K(k)^4}{10L^4}+\frac{1}{10L^4}.
\end{equation*}
The question about the stability of these periodic waves can be treated in a similar way with the necessary modifications. 
\end{remark}

Next, we will obtain the spectral properties related to the operator $\mathcal{L}=\partial^4_x+\omega-\phi^2$ as required in $(H)$. To do so, we will  utilize the following result of \cite{AN}:

\begin{proposition}\label{ANtheorem}
	Suppose that $\phi$ is a positive even solution of (\ref{edokawa}) such that $\widehat{\phi}>0$ and $\frac{\partial^2}{{\partial x}^2}(\log{g(x)})<0$, $x\neq0$, where $g$ is a real function such that $g(n)=\widehat{\phi}(n)$, $n\geq0$. Then the operator $\mathcal{L}$ has only one negative eigenvalue which is simple and zero is a simple eigenvalue whose eigenfunction is $\phi'$.
\end{proposition}
\begin{proof}
The proof of this result can be found in \cite[Theorem 4.1]{AN}. See also \cite{albert1} for the continuous case.
\end{proof}

\indent We are able to prove our stability result.

\begin{theorem}\label{mainT}
The periodic waves in $(\ref{solkawa})$  are orbitally stable in the sense of Definition $\ref{defstab}$.
\end{theorem}
\begin{proof}
In fact, according with \cite{ayse}, solution $\phi$ in (\ref{solkawa}) has the Fourier expansion
$
\phi(x) = a+\sum_{n=1}^{\infty}n\Gamma \mbox{csch}\left(\frac{n\pi K(k')}{K(k)}\right)\cos\left(\frac{2\pi n}{L}x\right), 
$
where $\Gamma=\frac{b\pi^2}{K(k)^2}$ and $k'=\sqrt{1-k^2}$.
Therefore, the Fourier coefficients of $\phi$ are given by
\begin{equation*}
\widehat{\phi}(n)=\left \{
\begin{array}{cc}
a, & n=0 \\
\frac{\Gamma}{2} n\mbox{csch}\left(\frac{n\pi K(k')}{K(k)}\right), & n\neq0. \\
\end{array}
\right.
\end{equation*}

By considering $g_k(x):=\frac{\Gamma}{2} x\mbox{csch}\left(\frac{x\pi K(k')}{K(k)}\right)=\frac{12\sqrt{10}\pi^2}{L^2} x\mbox{csch}\left(\frac{x\pi K(k')}{K(k)}\right)$, $x\in\mathbb{R}$, it is possible to see that  $\frac{\partial^2}{{\partial x}^2}(\log{g_k(x)})<0$, for all $x\in\R$ and $k\in(0,1)$ (see Figure 3.1).

\begin{figure}[!htb]
	\includegraphics[scale=0.4]{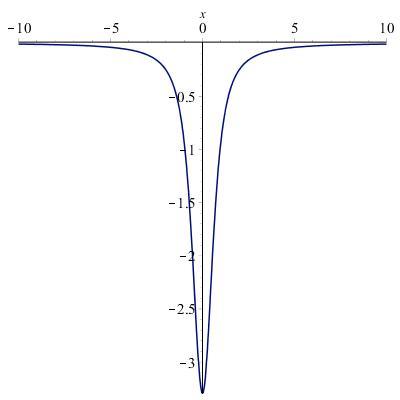}
	\caption{Graph of $\frac{\partial^2}{{\partial x}^2}(\log{g(x)})$ with $k=\frac{\sqrt{2}}{2}$.}
\end{figure}

So, using  Proposition \ref{ANtheorem}, we obtain that $\mathcal{L}$ has only one negative eigenvalue which is simple and zero is a simple eigenvalue whose eigenfunction is $\phi'$. Therefore, we have that $(H)$ holds. 

The next step is to find $\Phi$ as in Lemma $\ref{prop2}$. In fact, we consider $\Phi=\frac{\partial}{\partial k}\phi$. Note that, using \eqref{edokawa}, we have 
\begin{equation}\label{I1}
\mathcal{L}\Phi=-\frac{\partial\omega}{\partial k} \phi- \frac{\partial A}{\partial k}.  
\end{equation}
Equality $(\ref{I1})$ inspires us from $(\ref{functQ})$ in choosing $\nu=\frac{\partial \omega}{\partial k}$ and $\mu=\frac{\partial A}{\partial k}$ to get $Q(u)=\frac{\partial\omega}{\partial k} F(u)+\frac{\partial A}{\partial k}M(u)$ and $\mathcal{L}\Phi=-Q'(\phi)$. Thus, we have $\langle \mathcal{L}\Phi,\varphi\rangle=-\langle Q'(\phi),\varphi\rangle=0$, for all $\varphi\in\Upsilon_0$ and
$$
\mathcal{I}=\langle\mathcal{L}\Phi,\Phi\rangle=-\frac{\partial\omega}{\partial k}\frac{\partial}{\partial k}F(\phi)-\frac{\partial A}{\partial k}\frac{\partial}{\partial k}M(\phi).
$$
Using the explicit form \eqref{solkawa}, it is possible to deduce that $M(\phi)=aL$ and
$F(\phi)=\frac{480 K(k)^3}{L^4}\left[k^2L+\frac{2L(k^4-k^2+1)}{3}K(k)\right].$
Therefore, we are able to obtain 
\begin{equation}\label{Lkawa}
\mathcal{I}=\langle\mathcal{L}\Phi,\Phi\rangle=-\frac{1}{L^7}f(k),
\end{equation}
where $f$ is a complicated positive function depending smoothly on $k\in(0,1)$ (see Figure 3.2). This proves the required in Lemma $\ref{prop2}$. From Theorem \ref{teoest}, we conclude that $\phi$ is orbitally stable in $H_{per}^2([0,L])$ by the periodic flow of (\ref{mkawa}).

\begin{figure}[!htb]
	\includegraphics[scale=0.4]{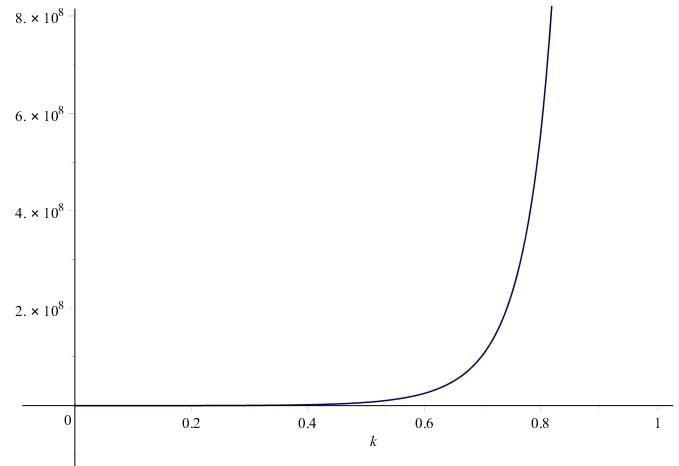}
\caption{Graphic of $f(k)$ in $(\ref{Lkawa})$.}
\end{figure}

\end{proof}

\section*{Acknowledgements}

F. C. is supported by FAPESP/Brazil grant 2017/20760-0. F. N. is supported by Funda\c{c}\~ao Arauc\'aria/Brazil, CNPq/Brazil and CAPES/Brazil.

\end{document}